\documentclass[a4paper,12pt]{article}

\usepackage{amsmath,amssymb,amsthm, mathrsfs}
\usepackage{latexsym}
\usepackage{graphics}
\usepackage{hyperref}
\usepackage{graphicx,psfrag}
\usepackage[bf, small]{titlesec}
\usepackage{authblk} 

\usepackage{lineno}

\theoremstyle{plain}
\newtheorem{thm}{Theorem}[section]
\newtheorem{lem}[thm]{Lemma}
\newtheorem{prop}[thm]{Proposition}
\newtheorem{cor}[thm]{Corollary}

\newtheorem{conj}[thm]{Conjecture}

\theoremstyle{definition}
\newtheorem{defi}[thm]{Definition}
\newtheorem{rem}[thm]{Remark}

\usepackage{standalone}
\usepackage{pgfpages,tikz,pgfkeys,pgfplots}
\usetikzlibrary{arrows,positioning,matrix,fit,backgrounds,shapes,shapes.geometric, intersections}

\usetikzlibrary{shapes.callouts,decorations.pathmorphing} 
\usetikzlibrary{shadows} 
\usepackage{calc} 
\usetikzlibrary{decorations.markings} 
\tikzstyle{vertex}=[circle, draw, inner sep=0pt, minimum size=6pt] 
\newcommand{\vertex}{\node[vertex]}
\usetikzlibrary{arrows,matrix} 

\newcommand{\AAA}{\mathcal{A}} 
\newcommand{\CCC}{\mathcal{C}} 
\newcommand{\FFF}{\mathcal{F}} 
\newcommand{\HHH}{\mathcal{H}} 

\title{The graph grabbing game on $\{0,1\}$-weighted graphs}

\author[1]{Soogang Eoh} 
\author[2]{Jihoon Choi
\thanks{Corresponding author: jihoon@cju.ac.kr}}

\affil[1]{Department of Mathematics Education,
Seoul National University, Seoul 08826, Republic of Korea}
\affil[2]{Department of Mathematics Education,
Cheongju University, Cheongju 28503, Republic of Korea}

\begin{document}
\maketitle
\begin{abstract}
The \emph{graph grabbing game} is a two-player game on a weighted connected graph in which two players, Alice and Bob, alternatively remove non-cut vertices one by one to gain the weights on them.
Alice wins the game if she gains at least half of the total weights.
In this paper, we show that on every connected even graph which does contain a fully spiked cycle as an induced subgraph, Alice always has a winning strategy with an arbitrary weight function whose codomain is $\{0,1\}$.
In addition, we give a list of forbidden subgraph for the family of graphs on which Alice has a winning strategy with an arbitrary weight function whose codomain is $\{0,1\}$.
\end{abstract}


\section{Introduction}

In this paper, every graph is assumed to be simple and finite unless otherwise stated.
A \emph{weight function} on a graph $G$ is a map $w$ from $V(G)$ to a set of real numbers.
A graph $G$ with a weight function $w$ is called a \emph{weighted graph}.
In this paper, we only deal with weighted graphs, so we sometimes omit ``weighted''.

For a connected graph $G$, a vertex $x$ is called a \emph{cut vertex} if $G - x$ is disconnected, and called a \emph{non-cut vertex} otherwise.
An \emph{even} (resp.\ \emph{odd}) graph is a graph having an even (resp.\ odd) number of vertices.

The \emph{graph grabbing game} is a two-player game on a weighted connected graph.
Two players, Alice and Bob, alternatively remove a vertex in each of their turn and take the weight on it.
The game rule is that Alice is the starting player, and in each turn, the player must take one of the non-cut vertices of the current graph so that the remaining vertices still form a connected graph.
After all the vertices are taken away, Alice wins the game if and only if she gains at least half of the total weight.

The graph grabbing game was introduced by Winkler~\cite{winkler2003mathematical}.
In his book, he showed that Alice has a winning strategy on every even path and that she does not have a winning strategy on some odd path.
Knauer \emph{et al}.~\cite{knauer2011eat} proved that Alice can always gain at least half of the total weight on every even cycle.
Micek and Walczak~\cite{micek2011graph} showed that Alice guarantees at least quarter of the total weight on every even tree.
Then they gave a conjecture that Alice always has a winning strategy on every even tree.
This conjecture was solved by Seacrest and Seacrest~\cite{seacrest2012grabbing} and then they conjectured that Alice can win the game on every connected bipartite even graph.
Egawa \emph{et al}.~\cite{egawa2018graph} supported this conjecture by showing that Alice can win the game on every even $K_{m,n}$-tree.

Now we introduce the following notations and restate the previous results.

\begin{defi}\label{def:aaa}
For a positive integer $k$, let $\AAA_k$ be the set of connected graphs on which Alice has a winning strategy with an arbitrary weight function whose codomain consists of $k$ real numbers.
Let $\AAA = \bigcap_{k=1}^\infty \AAA_k$.
\end{defi}
We note that $\AAA_1 \supset \AAA_2 \supset \AAA_3 \supset \cdots$.
Therefore $\AAA$ is the set of connected graphs on which Alice has a winning strategy with an arbitrary weight function.
We also note that $\AAA_1$ is the set of connected graphs and $\AAA_2$ is equal to the set of connected graphs on which Alice has a winning strategy with an arbitrary weight function whose codomain is $\{0,1\}$.
It is easy to see that Alice cannot win the game on the graph in Figure~\ref{fig:P3} as long as Bob plays optimally, which implies $\AAA_1 \supsetneq \AAA_2$.
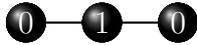
\begin{figure}
\begin{center}
\begin{tikzpicture}[auto,thick]
    \tikzstyle{player}=[minimum size=15pt,inner sep=0pt,outer sep=0pt,ball color=black,circle]

{

    \path[white] (0,0)   node [player]   (a) {\bf $0$};
    \path[white] (1,0)   node [player]  (b) {\bf $1$};
    \path[white] (2,0)   node [player]  (c) {\bf $0$};
}

{
   \draw[black,thick,-] (a) -- ++(b) ;
   \draw[black,thick,-] (b) -- ++(c) ;}

    \end{tikzpicture}
\end{center}
\caption{A graph $P_3$ and a weight function $w: V(P_3) \to \{0,1\}$} 
\label{fig:P3}
\end{figure}


\begin{defi}\label{def:hhh}
For a positive integer $k$, let $\HHH_k$ be the set of graphs $G$ satisfying the property that  every connected subgraph of $G$ induced by a vertex subset of an even size belongs to $\AAA_k$.
Let $\HHH = \bigcap_{k=1}^\infty \HHH_k$.
\end{defi}
We note that $\HHH_1 \supset \HHH_2 \supset \HHH_3 \supset \cdots$.
Therefore $\HHH$ is the set of graphs $G$ satisfying the property that, on every connected subgraph of $G$ induced by a vertex subset of an even size, Alice has a winning strategy with an arbitrary weight function.
By Definitions~\ref{def:aaa} and~\ref{def:hhh}, if a connected even graph belongs to $\HHH_k$ (resp.\ $\HHH$),  then it belongs to $\AAA_k$ (resp.\ $\AAA$).
We also note that $\HHH_1$ is the set of graphs.
Yet, there is a graph not belonging to $\HHH_2$ (which will be presented in Figure~\ref{fig:asteroidal triple}), which implies $\HHH_1 \supsetneq \HHH_2$.

It is clear that, for a positive integer $k$, the property that a graph belongs to $\HHH_k$ is hereditary, that is, for any graph in $\HHH_k$, all of its induced subgraphs belong to $\HHH_k$.
Therefore the property for a graph to be in $\HHH$ is also hereditary.

Following our terminology, the results of \cite{knauer2011eat}, \cite{seacrest2012grabbing}, and \cite{egawa2018graph} can be restated as follows:

\begin{thm}[\cite{knauer2011eat}, \cite{seacrest2012grabbing},  \cite{egawa2018graph}]\label{thm:previous results}
All the cycles and $K_{m,n}$-trees belong to $\HHH$.
\end{thm}

Theorem~\ref{thm:previous results} supports the conjecture given by Seacrest and Seacrest~\cite{seacrest2012grabbing}, which can be restated as follows:

\begin{conj}[\cite{seacrest2012grabbing}]\label{conj:bipartite}
Every bipartite graph belongs to $\HHH$.
\end{conj}

In this paper, we put our focus on the family $\HHH_2$.
We will show that every graph which does not contain any ``fully spiked cycle'' as an induced subgraph belongs to $\HHH_2$, which consequently asserts that every interval graph belongs to $\HHH_2$.
Then we present a list of forbidden subgraphs for $\HHH_2$.

\section{Main Results: On the family $\HHH_2$}

In this section, we study on the family $\HHH_2$.
We use the following notations. On a connected graph $G$ with $2k$ vertices, we denote by $a_i$ and $b_i$ the vertices which Alice and Bob take in their $i$th turn, respectively, for $i=1,\ldots,k$.
In addition, we assume that every weight function has the codomain $\{0,1\}$.

\begin{defi}
Let $G$ be a connected graph with at least three vertices, $x$ be a leaf, and $y$ be the unique neighbor of $x$ in $G$.
If $y$ is not a cut vertex of $G-x$, then we call $x$ a \emph{spike} in $G$.
\end{defi}

For a connected graph $G$, we denote the set of non-cut vertices of $G$ by $\Omega(G)$.

\begin{prop}\label{prop:omega subset}
For a connected graph $G$ with at least three vertices and a non-cut vertex $x$ of $G$, $\Omega(G-x) \subset \Omega(G)$ if and only if $x$ is not a spike in $G$.
\end{prop}
\begin{proof}
($\Rightarrow$) Suppose that $x$ is a spike in $G$.
Let $y$ be the neighbor of $x$ in $G$.
Then $y \notin \Omega(G)$.
However, since $x$ is a spike, $y \in \Omega(G-x)$.
Therefore $\Omega(G-x) \not\subset \Omega(G)$.

($\Leftarrow$) Suppose that $x$ is not a spike in $G$.
If $\Omega(G) = V(G)$, then we are done.
Suppose $\Omega(G) \neq V(G)$, i.e, $G$ has a cut vertex.
Take a cut vertex $z$ in $G$.
Let $C_1, C_2, \ldots, C_k$ ($k \ge 2$) be the components of $G-z$.
Without loss of generality, we may assume $x \in V(C_1)$.
Assume $|V(C_1)| \ge 2$.
Then $V(C_1 - x) \neq \emptyset$ and so $(G-x)-z$ has at least $k$ components.
By the way, $G-x$ is connected by the hypothesis that $x$ is a non-cut vertex.
Therefore $z$ is a cut-vertex in $G-x$.
Assume $|V(C_1)| = 1$, i.e., $V(C_1) = \{x\}$.
Then $z$ is the unique neighbor of $x$ and so $x$ is a leaf in $G$.
Since $x$ is not a spike in $G$ and $G$ has at least three vertices, $z$ is a cut-vertex in $G-x$.
Therefore $\Omega(G-x) \subset \Omega(G)$.
\end{proof}

For an integer $n \ge 3$, the \emph{fully spiked $n$-cycle}, denoted by $C^*_n$, is defined to be the graph obtained from the cycle $C_n$ by attaching a leaf to each vertex of $C_n$, that is, $V(C_n^*) = \{x_1, \ldots, x_n, y_1, \ldots, y_n\}$ and $E(C_n^*) = \{x_1x_2, x_2x_3, \ldots, x_{n-1}x_n, x_nx_1\} \cup \{x_iy_i \mid i=1,\ldots,n\}$ (see Figure~\ref{fig:spiked cycles}).
If $n$ is even (resp.\ odd), then $C^*_n$ is called a fully spiked \emph{even} (resp.\ \emph{odd}) cycle.
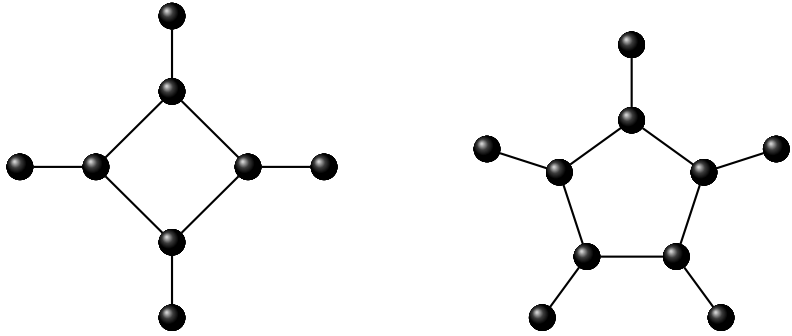
\begin{figure}
\begin{center}
\begin{tikzpicture}[rotate=90][auto,thick]
    \tikzstyle{player}=[minimum size=10pt,inner sep=0pt,outer sep=0pt,ball color=black,circle]

{
    \path[white] (0:1cm)    node [player]   (x1) {\bf $$};
    \path[white] (0:2cm)    node [player]   (y1) {\bf $$};
    \path[white] (90:1cm)   node [player]   (x2) {\bf $$};
    \path[white] (90:2cm)   node [player]   (y2) {\bf $$};
    \path[white] (90*2:1cm)   node [player]  (x3) {\bf $$};
    \path[white] (90*2:2cm)   node [player]  (y3) {\bf $$};
    \path[white] (90*3:1cm)   node [player]  (x4) {\bf $$};
    \path[white] (90*3:2cm)   node [player]  (y4) {\bf $$};
}

{
   \draw[black,thick,-] (x1) -- ++(y1) ;
   \draw[black,thick,-] (x2) -- ++(y2) ;
   \draw[black,thick,-] (x3) -- ++(y3) ;
   \draw[black,thick,-] (x4) -- ++(y4) ;

   \draw[black,thick,-] (x1) -- ++(x2) ;
   \draw[black,thick,-] (x2) -- ++(x3) ;
   \draw[black,thick,-] (x3) -- ++(x4) ;
   \draw[black,thick,-] (x4) -- ++(x1) ;
}

    \end{tikzpicture}
\qquad \qquad
\begin{tikzpicture}[rotate=90][auto,thick]
    \tikzstyle{player}=[minimum size=10pt,inner sep=0pt,outer sep=0pt,ball color=black,circle]

{
    \path[white] (0:1cm)    node [player]   (x1) {\bf $$};
    \path[white] (0:2cm)    node [player]   (y1) {\bf $$};
    \path[white] (72:1cm)   node [player]   (x2) {\bf $$};
    \path[white] (72:2cm)   node [player]   (y2) {\bf $$};
    \path[white] (72*2:1cm)   node [player]  (x3) {\bf $$};
    \path[white] (72*2:2cm)   node [player]  (y3) {\bf $$};
    \path[white] (72*3:1cm)   node [player]  (x4) {\bf $$};
    \path[white] (72*3:2cm)   node [player]  (y4) {\bf $$};
    \path[white] (72*4:1cm)   node [player]  (x5) {\bf $$};
    \path[white] (72*4:2cm)   node [player]  (y5) {\bf $$};

}

{
   \draw[black,thick,-] (x1) -- ++(y1) ;
   \draw[black,thick,-] (x2) -- ++(y2) ;
   \draw[black,thick,-] (x3) -- ++(y3) ;
   \draw[black,thick,-] (x4) -- ++(y4) ;
   \draw[black,thick,-] (x5) -- ++(y5) ;

   \draw[black,thick,-] (x1) -- ++(x2) ;
   \draw[black,thick,-] (x2) -- ++(x3) ;
   \draw[black,thick,-] (x3) -- ++(x4) ;
   \draw[black,thick,-] (x4) -- ++(x5) ;
   \draw[black,thick,-] (x5) -- ++(x1) ;
}

    \end{tikzpicture}
\end{center}
\caption{Fully spiked cycles $C_4^*$ and $C^*_5$}
\label{fig:spiked cycles}
\end{figure}
A graph is said to be \emph{$C^*$-free} if it does not contain the fully spiked $n$-cycle as an induced subgraph for any integer $n \ge 3$.

\begin{lem}\label{lem:spiked-free non-cut vertex}
For a connected $C^*$-free graph $G$, every cycle of $G$ contians a non-cut vertex of $G$.
\end{lem}
\begin{proof}
Suppose, to the contrary, that $G$ has a cycle $C$ every vertex on which is a cut-vertex of $G$.
Then there exists a chordless cycle, say $H := v_1v_2\cdots v_nv_1$ ($n \ge 3$), such that $V(H) \subset V(C)$.
For each $i = 1,\ldots,n$, let $X_i$ be the component of $G-v_i$ containing $V(H-v_i)$ and let $w_i \notin V(X_i)$ be a neighbor of $v_i$ in $G$.
For convenience, let $V = V(H)$ and $W = \{w_1, \ldots, w_n\}$.
Then $(V \setminus \{v_i\}) \cup (W \setminus \{w_i\}) \subset X_i$ for each $i=1,\ldots,n$.
In addition, since $w_i \notin V(X_i)$, $w_i$ is neither equal to nor adjacent to any vertex in $(V \setminus \{v_i\}) \cup (W \setminus \{w_i\})$ for any $i=1,\ldots,n$.
Therefore the subgraph of $G$ induced by $V \cup W$ is isomorphic to the fully spiked cycle $C^*_n$, 
which contradicts the hypothesis that $G$ is $C^*$-free.
\end{proof}


The following theorem is one of the main results of this paper.

\begin{thm}\label{thm:spiked-free graphs belongs to HHH2}
Every $C^*$-free graph belongs to $\HHH_2$.
\end{thm}
\begin{proof}
Let $G$ be a $C^*$-free graph.
To show that $G \in \HHH_2$, take a connected subgraph $H$ of $G$ induced by a vertex subset of an even size.
Clearly $H$ is $C^*$-free.
We prove that Alice has a winning strategy on $H$ by induction on $|V(H)|$.
If $|V(H)|=2$, then Alice certainly has a winning strategy on $H = K_2$.
Assume that Alice always has a winning strategy on $H$ when $|V(H)| \le 2k$ for some positive integer $k$.
Suppose $|V(H)| = 2k+2$.
Let $w : V(H) \to \{0,1\}$ be an arbitrary weight function on $H$.

If $H$ is a tree, then Alice has a winning strategy on $H$ by Theorem~\ref{thm:previous results} and we are done.
Suppose that $H$ is not a tree.

\medskip

{\bf Case 1.} There is a vertex in $\Omega(H)$ of weight $1$.

Alice takes a vertex in $\Omega(H)$ of weight $1$ as $a_1$.
Then $w(b_1) \le 1 = w(a_1)$.
The graph $H - \{a_1,b_1\}$ is obviously $C^*$-free and still connected by the game rule.
In addition, $H - \{a_1,b_1\}$ is a subgraph of $G$ induced by $2k$ vertices.
Therefore, the induction hypothesis tells us that Alice has a winning strategy on $H - \{a_1,b_1\}$, i.e., $\sum_{i=2}^{k+1} w(a_i) \ge \sum_{i=2}^{k+1} w(b_i)$.
Thus $\sum_{i=1}^{k+1} w(a_i) \ge \sum_{i=1}^{k+1} w(b_i)$ and Alice can win the game on $H$.

\medskip

{\bf Case 2.} Every vertex in $\Omega(H)$ has weight $0$.

Since $H$ is not a tree, $H$ has a cycle.
By Lemma~\ref{lem:spiked-free non-cut vertex}, $H$ has a non-cut vertex $x$ on the cycle.
Alice takes $x$ as $a_1$.
Since $x$ is a vertex on a cycle, $x$ is not a spike in $H$ and so, by Proposition~\ref{prop:omega subset}, $\Omega(H-x) \subset \Omega(H)$.
By the case assumption, every vertex in $\Omega(H-x)$ has weight $0$, so $w(b_1)=0$.
Again, we apply the induction hypothesis on $H - \{a_1,b_1\}$ to conclude that Alice wins the game on $H$.

Hence $G \in \HHH_2$ and this completes the proof.
\end{proof}

\begin{cor}
Every interval graph belongs to $\HHH_2$.
\end{cor}
\begin{proof}
Let $G$ be an interval graph.
By a well-known property that an interval graph is $C_3^*$-free, $G$ is $C^*_3$-free.
In addition, since no interval graph contains an induced cycle of length at least four, $G$ is $C^*_n$-free for any $n \ge 4$.
Hence the corollary follows from Theorem~\ref{thm:spiked-free graphs belongs to HHH2}.
\end{proof}

For a family $\FFF$ of graphs, a graph $H$ is said to be \emph{forbidden for $\FFF$} if no graph containing $H$ as an induced subgraph belongs to $\FFF$.

For a positive integer $k$, let $\CCC_k$ be the set of graphs which are forbidden for $\HHH_k$.
Then $\CCC_1 \subset \CCC_2 \subset \CCC_3 \subset \cdots$.
We note that $\CCC_1 = \emptyset$ and that $\bigcup_{k=1}^\infty \CCC_k$ is the set of graphs forbidden for $\HHH$.
Theorem~\ref{thm:spiked-free graphs belongs to HHH2} tells us that no $C^*$-free graph belongs to $\CCC_2$.
Then it would be interesting to ask whether the fully spiked cycles belong to $\CCC_2$.

\begin{prop}\label{prop:even cycles not forbidden}
No fully spiked even cycle belongs to $\CCC_2$.
\end{prop}
\begin{proof}
Let $k \ge 4$ be a positive even integer.
It suffices to show that $C^*_k \in \HHH_2$.
Let $w: V(C^*_k) \to \{0,1\}$ be an arbitrary weight function.
We note that whatever Alice and Bob take in their first turn, $C^*_k - \{a_1,b_1\}$ is $C^*$-free and so Alice can win the game on $C^*_k - \{a_1,b_1\}$ by Theorem~\ref{thm:spiked-free graphs belongs to HHH2}.

\medskip

{\bf Case 1.} $|\{x \in V(C^*_{k}) \mid w(x)=1\}|$ is even

Suppose $w(a_1) \ge w(b_1)$.
Since Alice can win the game on $C^*_k - \{a_1,b_1\}$, $\sum_{i=2}^k w(a_i) \ge \sum_{i=2}^k w(b_i)$.
Then, by the supposition that $w(a_1) \ge w(b_1)$, $\sum_{i=1}^k w(a_i) \ge \sum_{i=1}^k w(b_i)$ and so Alice wins the game on $C^*_k$.

Suppose $w(a_1) < w(b_1)$.
Then $w(a_1)=0$ and $w(b_1)=1$.
Since Alice can win the game on $C^*_k - \{a_1,b_1\}$ and since $|\{x \in V(C^*_{k}) \setminus \{a_1,b_1\} \mid w(x)=1\}|$ is odd, $\sum_{i=2}^k w(a_i) \ge 1 + \sum_{i=2}^k w(b_i)$.
Thus $\sum_{i=1}^k w(a_i) \ge \sum_{i=1}^k w(b_i)$ and so Alice wins the game on $C^*_k$.

\medskip

{\bf Case 2.} $|\{x \in V(C^*_{k}) \mid w(x)=1\}|$ is odd

Suppose that there is a leaf of weight $1$.
Alice takes one of the leaves as $a_1$.
Then $w(a_1) \ge  w(b_1)$.
Since Alice can win the game on $C^*_k - \{a_1,b_1\}$, she can win the game on $C^*_k$.

Suppose that every leaf has weight $0$.
Since $k$ is even and $|\{x \in V(C^*_{k}) \mid w(x)=1\}|$ is odd, there is a vertex on a cycle of weight $0$.
Alice takes a leaf whose neighbor has weight $0$ as $a_1$.
Then $w(a_1) = 0 = w(b_1)$ and so Alice wins the game on $C^*_k$.
\end{proof}

\begin{prop}\label{prop:odd cycles forbidden}
Every fully spiked odd cycle belongs to $\CCC_2$.
\end{prop}
\begin{proof}
Let $k$ be a positive integer and let $w: V(C^*_{2k+1}) \to \{0,1\}$ be a weight function on $C^*_{2k+1}$ defined so that all the leaves have weight $0$ and the other vertices have weight $1$.
It suffices to show that Bob has a winning strategy on $C^*_{2k+1}$ with the weight function $w$ in order to assert that $C^*_{2k+1} \in \CCC_2$.

Alice must choose one of the leaves of $C^*_{2k+1}$ as $a_1$.
Then Bob chooses the unique neighbor of $a_1$ as $b_1$.
Note that Bob is leading the game by one point so far.
We label the vertices of the graph $C^*_{2k+1} - \{a_1,b_1\}$ as given in Figure~\ref{fig:odd cycle minus a1 b1}.
\begin{figure}
\centering
  \begin{tikzpicture}[x=1.2cm, y=1.2cm]

    \vertex (x0) at (0,1) [label=above:$x_1$]{};
    \vertex (x1) at (1,1) [label=above:$x_2$]{};
    \vertex (y1) at (2,1) [label=above:$x_3$]{};
    \vertex (x2) at (3,1) [label=above:$x_4$]{};
    \vertex (y2) at (4,1) [label=above:$x_5$]{};
    \vertex (x3) at (5,1) [label=above:$x_{2k-2}$]{};
    \vertex (y3) at (6,1) [label=above:$x_{2k-1}$]{};
    \vertex (x4) at (7,1) [label=above:$x_{2k}$]{};

    \vertex (v0) at (0,0) [label=below:$y_1$]{};
    \vertex (v1) at (1,0) [label=below:$y_2$]{};
    \vertex (w1) at (2,0) [label=below:$y_3$]{};
    \vertex (v2) at (3,0) [label=below:$y_4$]{};
    \vertex (w2) at (4,0) [label=below:$y_5$]{};
    \vertex (v3) at (5,0) [label=below:$y_{2k-2}$]{};
    \vertex (w3) at (6,0) [label=below:$y_{2k-1}$]{};
    \vertex (v4) at (7,0) [label=below:$y_{2k}$]{};

    \node (z1) at (4.5,1) [label=center:$\cdots$]{};

    \path

    (x0) edge [-,thick] (x1)
    (x1) edge [-,thick] (y1)
    (y1) edge [-,thick] (x2)
    (x2) edge [-,thick] (y2)
    (x3) edge [-,thick] (y3)
    (y3) edge [-,thick] (x4)

    (y2) edge [-,thick,shorten >=24] (x3)
    (x3) edge [-,thick,shorten >=23] (y2)

    (x0) edge [-,thick] (v0)
    (x1) edge [-,thick] (v1)
    (x2) edge [-,thick] (v2)
    (x3) edge [-,thick] (v3)
    (x4) edge [-,thick] (v4)

    (y1) edge [-,thick] (w1)
    (y2) edge [-,thick] (w2)
    (y3) edge [-,thick] (w3)

	;

\end{tikzpicture}
  \caption{A graph obtained from $C^*_{2k+1}$ by deleting a leaf and its neighbor}
  \label{fig:odd cycle minus a1 b1}
\end{figure}
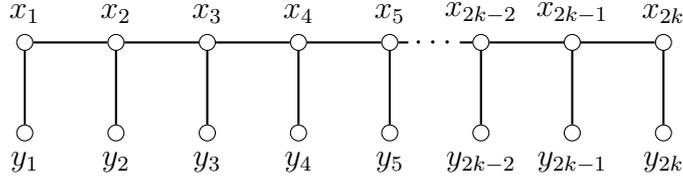
According to the all possible choices of Alice, Bob plays the game in the following rules:
\begin{itemize}


\item[(R1)] If Alice takes $y_{2i-1}$ (resp.\ $y_{2i}$) in her turn for some $i \in \{1,\ldots,k\}$, then Bob immediately takes $y_{2i}$ (resp.\ $y_{2i-1}$) in his next turn.

\item[(R2)] If Alice takes $x_{2i-1}$ (resp.\ $x_{2i}$) in her turn for some $i \in \{1,\ldots,k\}$, then Bob immediately takes $x_{2i}$ (resp.\ $x_{2i-1}$) in his next turn.
\end{itemize}


Now we explain why Bob can obey the rules.
It is clear that Bob can obey the rule (R1) since $y_1, \ldots, y_{2k}$ are leaves.
Suppose that Bob had successfully obeyed the rule (R2) until the $(j-1)$st turn for some $j \in \{1,\ldots,k\}$.
Assume that Alice took either $x_{2i-1}$ or $x_{2i}$ for some $i \in \{1,\ldots,k\}$ as $a_{j+1}$.
By symmetry, we may assume that $a_{j+1}=x_{2i-1}$.
Since Bob had obeyed the rules (R1) and (R2) until the $j$th turn, $x_{2i}$ have not been taken by anyone.
Then Alice must have taken $a_{j+1}=x_{2i-1}$ as a leaf and this implies that $y_{2i-1}$ had been taken before the $(j+1)$st stage.
Therefore, by (R1), $y_{2i}$ had also been taken before the $(j+1)$st stage, so $x_{2i}$ becomes a leaf or the last vertex after Alice's $(j+1)$st turn.
Thus Bob can take $x_{2i}$ as $b_{j+1}$ in his $(j+1)$st turn to obey the rule (R2).
Hence Bob can always obey the rules (R1) and (R2) in every turn.

As a result, $w(a_i) = w(b_i)$ for each $i=2,\ldots,k+1$.
Since $w(a_1) < w(b_1)$, Bob wins the game by following this strategy.
\end{proof}


\begin{figure}
\begin{center}
\begin{tikzpicture}[rotate=90][auto,thick]
    \tikzstyle{player}=[minimum size=15pt,inner sep=0pt,outer sep=0pt,ball color=black,circle]

{
    \path[white] (0:1cm)    node [player]   (a) {\bf $1$};
    \path[white] (0:2cm)    node [player]   (d) {\bf $0$};
    \path[white] (120:1cm)   node [player]   (b) {\bf $1$};
    \path[white] (120:2cm)   node [player]   (e) {\bf $0$};
    \path[white] (120*2:1cm)   node [player]  (c) {\bf $1$};
    \path[white] (120*2:2cm)   node [player]  (f) {\bf $0$};
}

{
   \draw[black,thick,-] (a) -- ++(b) ;
   \draw[black,thick,-] (b) -- ++(c) ;
   \draw[black,thick,-] (c) -- ++(a) ;
   \draw[black,thick,-] (a) -- ++(d) ;
   \draw[black,thick,-] (b) -- ++(e) ;
   \draw[black,thick,-] (c) -- ++(f) ;
}

    \end{tikzpicture}
\end{center}
\caption{A graph $C_3^*$ and a weight function $w: V(C_3^*) \to \{0,1\}$} 
\label{fig:asteroidal triple}
\end{figure}
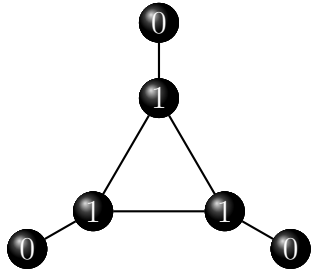
\begin{rem}
By Proposition~\ref{prop:odd cycles forbidden}, the fully spiked cycle $C_3^*$ illustrated in Figure~\ref{fig:asteroidal triple} does not belong to $\HHH_2$.
On the other hand, it is easy to see that any graph $G$ with at most five vertices is $C^*$-free and so, by Theorem~\ref{thm:spiked-free graphs belongs to HHH2}, $G \in \HHH_2$.
Therefore, among the graphs not in $\HHH_2$, $C_3^*$ is one with the smallest number of vertices.
By the way, $C_3^*$ is the only graph among the graphs with six vertices and containing a fully spiked cycle as an induced subgraph.
Thus $C_3^*$ is the only graph with the smallest number of vertices among the graphs not in $\HHH_2$.
\end{rem}

\section{Concluding remarks}

In this paper, we showed that a graph $G$ belongs to $\HHH_2$ if $G$ is $C^*$-free (Theorem~\ref{thm:spiked-free graphs belongs to HHH2}).
Then we presented the fully spiked odd cycles as forbidden subgraphs for $\HHH_2$ (Proposition~\ref{prop:odd cycles forbidden}).
This implies that a graph $G$ belongs to $\HHH_2$ only if $G$ is $C^*_n$-free for any odd integer $n \ge 3$.
We would like to ask whether the converse is true.

\begin{conj}\label{conj 3.1}
A graph $G$ belongs to $\HHH_2$ if and only if $G$ is $C^*_n$-free for any odd integer $n \ge 3$.
\end{conj}

To prove Conjecture~\ref{conj 3.1}, it suffices to show that every graph which contains a fully spiked even cycle as an induced subgraph but does not contain a fully spiked odd cycle belongs to $\HHH_2$, which is supported by Proposition~\ref{prop:even cycles not forbidden}.

Now we give another conjecture.

\begin{conj}\label{conj 3.2}
A graph $G$ belongs to $\HHH$ if and only if $G$ is $C^*_n$-free for any odd integer $n \ge 3$.
\end{conj}

If Conjecture~\ref{conj 3.2} is true, then Conjecture~\ref{conj:bipartite} is also true.

\section{Acknowledgement}
This research was supported by Basic Science Research Program through the National Research Foundation of Korea(NRF) funded by the Ministry of Education(NRF-2018R1D1A1B07049150).

\end{document}